%% file: s-bundles-I.tex
\documentclass{amsart}
\input{packages.tex} 
\input{notation.tex} 
\input{environments.tex} 

\newcommand{\flag}{Fl}
\newcommand{\cod}{\mathfrak{cod}}
\newcommand{\dom}{\mathfrak{dom}}
\newcommand{\gpd}[1]{\mathfrak{gpd}_{#1}}

\newcommand{\Unitarity}{Theorem 2.3, \cite{margolis1987inverse}}
\newcommand{\MacAllister}{MacAllister's Covering Theorem}
\newcommand{\Lawson}{Theorem 4.3, \cite{lawson1998inverse}}
\newcommand{\Schein}{Theorem 5, \cite{schein1996subsemigroups}}
\newcommand{\Isometries}{Proposition 2.1.4, \cite{paterson2012groupoids}}
\newcommand{\Semiorthogonal}{Proposition 2.1.4, \cite{paterson2012groupoids}}
\newcommand{\THMFiniteInverseReps}{Theorem 4.1, \cite{linckelmann2013inverse}}

\theoremstyle{plain}

\newtheorem*{thm:inverse-category-representations}{Theorem \ref{thm:inverse-category-representations}}
\newtheorem*{thm:inverse-extendability}{Theorem \ref{thm:inverse-extendability}}
\newtheorem*{cor:decomposability}{Corollary \ref{cor:decomposability}}
\newtheorem*{thm:inverse.structure}{Theorem \ref{thm:inverse.structure}}
\newtheorem*{thm:partial.isometries}{Theorem \ref{thm:partial.isometries}}
\newtheorem*{thm:inner.products}{Theorem \ref{thm:inner.products}}
\newtheorem*{thm:embed}{Theorem \ref{thm:embed}}
\newtheorem*{thm:decompose}{Theorem \ref{thm:decompose}}
\newtheorem*{thm:unitarity}{\Unitarity}
\newtheorem*{thm:macallister}{\MacAllister}
\newtheorem*{thm:lawson}{\Lawson}
\newtheorem*{thm:schein}{\Schein}
\newtheorem*{thm:isometries}{\Isometries}
\newtheorem*{thm:semiorthogonal}{\Semiorthogonal}
\newtheorem*{thm:metric}{Theorem \ref{thm:metric}}
\newtheorem*{thm:BS}{Theorem \ref{thm:BS}}
\newtheorem*{thm:mobius-inversion}{M\"{o}bius Inversion Theorem, \cite{rota1964foundations}}
\newtheorem*{thm:mobius-function}{Theorem 1.3, \cite{rota1964foundations}}
\newtheorem*{thm:finite.inverse.reps}{\THMFiniteInverseReps}
\newtheorem*{cor:covers}{Corollary \ref{cor:covers}}
\newtheorem*{prop:ehresmann-groupoids}{Proposition, \cite{dewolf2015ehresmann}}

\title{Invertibility in category representations}
\author{Sanjeevi Krishnan and Crichton Ogle}

\begin{document}
\begin{abstract}
  Inverse categories are categories in which every morphism $\zeta$ has a unique pseudo-inverse $\zeta^\dagger$ in the sense that $\zeta\zeta^\dagger\zeta=\zeta$ and $\zeta^\dagger\zeta\zeta^\dagger=\zeta^\dagger$.
  Persistence modules from topological data analysis and similarly decomposable category representations factor through inverse categories.    
  This paper gives a numerical condition, decidable when the indexing category is finite, characterizing when a representation of a small category factors through an inverse category.
\end{abstract}
\maketitle
\tableofcontents

\addtocontents{toc}{\protect\setcounter{tocdepth}{1}}

\section{Introduction}
Geometric data is often encoded in a diagram of vector spaces, a representation
\begin{equation}
  \label{eqn:diagrams}
  \nabla:\smallcat{1}\ra\VECTORSPACES_k
\end{equation}
of a small category $\smallcat{1}$ in the category $\VECTORSPACES_k$ of finite dimensional vector spaces over a field $k$.
Examples include \textit{$n$-dimensional persistence modules} in topological data analysis ($\smallcat{1}$ is a rank $n$ poset), quiver representations ($\smallcat{1}$ is a free category on a directed graph), cellular sheaves of vector spaces ($\smallcat{1}$ is a poset of cells in a cell complex) \cite{shepard1986cellular}, and parallel transport for a flat vector bundle ($\smallcat{1}$ is the fundamental groupoid of a manifold).
In order to better understand the data, we often wish to decompose the diagram $\nabla$ into a direct sum of simpler such diagrams.
All $1$-dimensional persistence modules and quiver representations over simply laced Dynkin diagrams decompose into what we call \textit{blockcodes}, isomorphisms on their supports, that are irreducible.    
These decompositions give a concise description of the data that is then amenable, say, to efficient storage and statistical analyses (eg. \cite{kwitt2015statistical}).
General representations of small categories, even those of interest in data analysis \cite{carlsson2009theory}, generally do not decompose into (irreducible) blockcodes. 
However, obstructions to decomposability into blockcodes are themselves informative invariants about the data.
It is therefore useful to describe special properties of direct sums of blockcodes.  
One is the existence of metric structure of some sort.

Metric structures make it possible to split off sub-representations of $\nabla$ when $k=\C$.  
In the case where $\smallcat{1}$ is a group, metric structure on $\nabla$ amounts to a lift to a category of inner product spaces and isometries. 
The existence of such a lift in this case is characterized by precompactness for the image of $\nabla$ when regarded as a homomorphism to $GL(n)$, because $U(n)$ is the maximal compact subgroup unique up to conjugacy in $GL(n)$.
An immediate consequence is that $\nabla$ decomposes into irreducible blockcodes when $\smallcat{1}$ is a finite group (Maschke's Theorem).  
In the general case, metric structure on $\nabla$ amounts to a lift to a category of inner product spaces and \textit{partial isometries}, isometries on orthogonal complements of kernels. 
Such a lift exists if $\smallcat{1}$ is a commutative, idempotent monoid  [Lemma \ref{lem:unitary.projections}].
Such a lift also exists if $\nabla$ is a $1$-dimensional persistence module.  
However, it is difficult to completely characterize the existence of such lifts (cf. \cite{mbekhta2009classes}.)
Part of the problem is that there is no clear analogue of $U(n)$: partial isometries do not generally compose to collectively form a single, concrete category (cf. \cite{hines2010structure}).  

Whenever some collection of partial isometries does happen to form the morphisms of some category, the closure of that category under adjoints has a special dagger structure.
An \textit{inverse category} $\smallcat{1}$ is a category in which every morphism $\zeta$ admits a unique \textit{pseudo-inverse} $\zeta^{\dagger}$ in the sense that $\zeta\zeta^\dagger\zeta=\zeta,\,\zeta^\dagger\zeta\zeta^\dagger=\zeta^\dagger$  \cite{hines2010structure}.
One class of examples consists of groupoids, say of vector spaces and isomorphisms. 
Another class of examples consists of adjoint-closed categories of inner product spaces and partial isometries.
For the special case $k=\C$, representations of inverse categories interpolate between diagrams in the first and second classes [Proposition \ref{prop:inverse-category-representation}].
Representations of finite inverse categories whose endomorphisms are all idempotent decompose into blockcodes, as a straightforward consequence of a larger structure theorem for finite inverse categories \cite[Theorem 4.1]{linckelmann2013inverse}.
Conversely, direct sums of blockcodes factor through inverse categories [Proposition \ref{prop:one-way}].
Mackey functors \cite{dress1971notes} can be characterized by their factorizability through certain inverse categories \cite[\S 6.4]{linckelmann2013inverse}. 
Thus motivated, the goal is to characterize factorizability of representations through inverse categories.

The idea behind the result [Theorem \ref{thm:inverse-extendability}] is that the salient structure of a representation $\nabla$ is abstractly captured by a diagram $\flag_\nabla$ [Definition \ref{defn:multiflag}] of posets, of subspaces generated by the operations of taking kernels, images, intersections, and inverse images in $\nabla$.
This diagram $\flag_{\nabla}$ functions as a generalized sort of flag, in a sense expounded elsewhere \cite{ogle2018structure} that a decomposition of $\nabla$ into simpler representations can be constructed as some sort of associated graded of $\flag_{\nabla}$ under a general position condition on $\flag_{\nabla}$ and the presence of metric structure on $\nabla$.
The factorizability problem of this paper can be numerically characterized as follows.  
Let $\mu_P$ denote the M\"{o}bius function $P\ra\Z$ of a finite poset $P$ having all infima, recursively defined by $\mu_P(\min\,P)=1$ and $\mu_P(y)=-\sum_{x<_Py}\mu_P(x)$ for all $y\neq\min\,P$.  
M\"{o}bius functions, used for generalized inclusion-exclusion counting formulas, are naturally used to decompose representations of, for example, finite inverse monoids \cite[Theorem 3.2]{steinberg2016representation}, more general finite inverse categories \cite[Theorem 4.1]{linckelmann2013inverse}, posets satisfying a certain finiteness condition \cite[Theorem 3.14]{kim2018generalized}, and constructible sheaves \cite[Theorem 4.1]{patel2018generalized}.
M\"{o}bius functions are used in this paper to characterize factorizability through inverse categories.  

\newcommand{\InverseExtendability}{
  The following are equivalent for a representation 
  $$\nabla:\smallcat{1}\ra\VECTORSPACES_k.$$
  \begin{enumerate}
    \item\label{item:factors} $\nabla$ factors through an inverse category
    \item\label{item:formula} For each $\smallcat{1}$-object $o$ and $\mathfrak{b},\mathfrak{c}\in\flag_\nabla(o)$, $\sum_{\mathfrak{a}\subset\mathfrak{b}}\mu_{\flag_{\nabla}(o)}(\mathfrak{a})(\dim\mathfrak{a}-\dim(\mathfrak{a}\cap\mathfrak{c}))\geqslant 0$.
  \end{enumerate}
}

\begin{thm:inverse-extendability}
  \InverseExtendability{}
\end{thm:inverse-extendability}

Representations of finite quivers without cycles that factor through inverse categories factor through finite inverse categories whose endomorphisms are all idempotent [Lemma \ref{lem:quivers}].
A consequence is the following sufficient condition for representations of such finite quivers without cycles to decompose into blockcodes.

\newcommand{\Decomposability}{
  The following are equivalent for a representation
  $$\nabla:\smallcat{1}\ra\VECTORSPACES_k$$
  of a free category $\smallcat{1}$ on a directed graph having no undirected cycles.
  \begin{enumerate}
    \item $\nabla$ factors through an inverse category
    \item $\nabla$ factors through a finite inverse category whose submonoids are all idempotent
    \item $\nabla$ is a direct sum of blockcodes
    \item For each $\smallcat{1}$-object $o$ and $\mathfrak{b},\mathfrak{c}\in\flag_{\nabla}(o)$, $\sum_{\mathfrak{a}\subset\mathfrak{b}}\mu_P(\mathfrak{a})(\dim\mathfrak{a}-\dim(\mathfrak{a}\cap\mathfrak{c}))\geqslant 0$.
  \end{enumerate}
}

\begin{cor:decomposability}
  \Decomposability{}
\end{cor:decomposability}

\subsection*{Organization}
Some conventions are established in \S\ref{sec:conventions}.
A categorical abstraction of pseudo-inverses for matrices is recalled in \S\ref{sec:pseudo-inverses}.  
Representations of categories are investigated in \S\ref{sec:representations}.
Some examples [Examples \ref{eg:trisection}, \ref{eg:bisection}] illustrate the tractability of the factorizability criterion.  

\subsection{Conventions}\label{sec:conventions}
This section fixes some conventions and recalls some basic definitions.  

\subsubsection{Categorical}
Let $\cat{1},\cat{2}$ denote categories.
Let $\smallcat{1},\smallcat{2}$ denote small categories.
Let $\ira$ denote an inclusion of some sort, such as an inclusion of a subset into a set or a subcategory into a category. 
Let $\id_{o}$ denote the identity morphism $o\ra o$ for an object $o$ in a given category.
Write $\dom(\zeta),\cod(\zeta)$ for the domain and codomain, respectively, of a morphism $\zeta$ in a given category.  
A \textit{monoid} will be regarded at once as a set equipped with an associative, unital multiplication as well as a small category with a unique object by identifying the former as the set of all (endo)morphisms of an associated category.  
An \textit{idempotent} in a monoid $M$ is a morphism $\zeta$ in $M$ such that $\zeta^2=\zeta$.  
A monoid is \textit{idempotent} if all of its morphisms are idempotent in it.  
Let $\CATS$ denote the category of all small categories and functors between them.  

\subsubsection{Order-theoretic}
Fix a poset $P$. 
Write $\leqslant_P$ for the partial order associated to $P$; write $x<_Py$ if $x\leqslant_Py$ and $x\neq p$.  
A poset $P$ will be identified with the small category whose arrows $x\ra y$ corresond to relations $x\leqslant_Py$.
Write $\min\,P$ for the (necessarily unique) minimum, if one exists, in $P$.
In the case $P$ is finite and has all infima, let $\mu_P$ denote the M\"{o}bius function $\mu_P:P\ra\Z$ on $P$, recursively defined by $\mu_P(y)=1$ if $y=\min\,P$ and $-\sum_{x<_Py}\mu_P(x)$ for $y\neq\min\,P$. 

\begin{thm:mobius-inversion}
  Consider the following data.
  \begin{enumerate}
    \item finite poset $P$ having all infima
    \item functions $\phi,\hat\phi:P\ra\Z$ such that $\hat\phi(y)=\sum_{x\leqslant_Py}\phi(x)$.
  \end{enumerate}
  Then $\phi(y)=\sum_{x\leqslant_Py}\hat\phi(x)\mu_P(y).$
\end{thm:mobius-inversion}

\subsubsection{Algebraic}
Let $k$ denote an arbitrary field. 
A \textit{ring} will be taken to mean a commutative ring.  
Let $\VECTORSPACES_k$ denote the category of finite-dimensional vector spaces over a field $k$ and linear maps between them.
Write ${\bf 0}$ for the trivial vector space, over a field understood from context.
Write $\mathfrak{a}\leqslant\mathfrak{b}$ if $\mathfrak{a}$ is a vector subspace of a vector space $\mathfrak{b}$ and $\mathfrak{a}<\mathfrak{b}$ if additionally $\mathfrak{a}\neq\mathfrak{b}$.  
Let $\C$ denote the complex field.  
Write $\im\,\phi$, $\kernel\,\phi$ for the image and kernel, respectively, of a linear map $\phi$.

\addtocontents{toc}{\protect\setcounter{tocdepth}{2}}

\section{Inverse categories}\label{sec:pseudo-inverses}
Fix a category $\cat{1}$.
Straightforwardly extending terminology for matrices, define a \textit{pseudo-inverse} to a $\cat{1}$-morphism $\alpha:x\ra y$ to be a $\cat{1}$-morphism $\beta:y\ra x$ such that $\beta\alpha\beta=\beta$ and $\alpha\beta\alpha=\alpha$. 
If $\beta$ is pseudo-inverse to $\alpha$ in $\cat{1}$, then $\alpha$ is pseudo-inverse to $\beta$ in $\cat{1}$ and $\alpha\beta,\beta\alpha$ are idempotent endomorphisms in $\cat{1}$.
Functors preserve pseudo-inverses.  
A category is \textit{inverse} if each of its morphisms admits a unique pseudo-inverse.
In an inverse category, denote the unique pseudo-inverse of a morphism $\zeta$ by $\zeta^\dagger$.

\begin{eg}
	Groupoids are inverse categories.  
\end{eg}

\begin{lem}
	\label{lem:inverse.products}
  Products of small inverse categories in $\CATS$ are inverse.
\end{lem}
\begin{proof}
	Consider a collection $\{X_i\}_{i\in\mathcal{I}}$ of small inverse categories, indexed by a set $\mathcal{I}$.
	Let $i,j$ denote elements in $\mathcal{I}$.
	Let $X=\prod_{i}X_i$.
	For each $j$, let $\pi_j$ denote the projection morphism of the form $X\ra X_j$. 
	
	For each pseudo-inverse $\zeta^*$ to a $X$-morphism $\zeta$, $\pi_j\zeta^*$ is pseudo-inverse to $\pi_j\zeta$ and hence $\pi_j\zeta^*=(\pi_j\zeta)^\dagger$.
	Conversely for each $X$-morphism $\zeta$, there exists a unique $X$-morphism $\zeta^*$ characterized by $\pi_j\zeta^*=(\pi_j\zeta)^{\dagger}$ for each $j$ and hence $\zeta^*\zeta\zeta^*=\zeta$ and $\zeta=\zeta\zeta^*\zeta$ because both equations hold after applying $\pi_j$, for each $j$.
	Thus $X$-morphisms have unique pseudo-inverses.
\end{proof}

\begin{lem}
	\label{lem:inverse.over.idempotents}
	Consider an inverse category $\cat{1}$ and functor
	$$F:\cat{1}\ra\cat{2}.$$
	For each $\cat{1}$-morphism $\zeta$ with $F(\zeta)$ an idempotent endomorphism, $F(\zeta)=F(\zeta^\dagger)$.
\end{lem}

The proof, almost word-for-word taken from a proof for the inverse monoid case \cite[Theorem 7.4]{clifford1967algebraic}, is included for completeness.

\begin{proof}
	For brevity, let $\alpha=F(\zeta)$, $\beta=F(\zeta^\dagger)$.  
	Then $\alpha$ is pseudo-inverse to $\beta$ and $$(\beta\alpha)(\alpha\beta)=F((\zeta^\dagger\zeta)(\zeta\zeta^\dagger))=F((\zeta\zeta^\dagger)(\zeta^\dagger\zeta))=(\alpha\beta)(\beta\alpha).$$
	Hence $\alpha=\alpha\beta\alpha=\alpha(\beta\alpha\beta)\alpha=\alpha(\beta\alpha^2\beta)\alpha=\alpha(\beta\alpha)(\alpha\beta)\alpha=\alpha(\alpha\beta)(\beta\alpha)\alpha=(\alpha\beta)(\beta\alpha)=(\beta\alpha)(\alpha\beta)=\beta\alpha\beta=\beta$.
\end{proof}

All endomorphism monoids of inverse categories are inverse monoids.  
In an inverse monoid $M$, every idempotent is of the form $xx^{\dagger}$.  
Inverse monoids are exactly the monoids whose morphisms admit pseudo-inverses and whose idempotents commute \cite[Theorem 5.1.1]{howie1995fundamentals}.  
The following characterization of inverse categories straightforwardly generalizes this latter characterization of inverse monoids.  

\begin{lem}
	\label{lem:inverse.characterization}
	The following are equivalent for a category $\cat{1}$.  
	\begin{enumerate}
		\item $\cat{1}$ is inverse.
		\item Each $\cat{1}$-morphism has a pseudo-inverse and the idempotents in each monoid $\cat{1}(o,o)$ commute for each $\cat{1}$-object $o$.
	\end{enumerate}
\end{lem}

The proof of (1)$\implies$(2) bootstraps the inverse monoid case \cite[Theorem 3.2]{steinberg2016representation}.
The proof of (2)$\implies$(1), almost word-for-word taken from a proof for (2)$\implies$(1) in the inverse monoid case (eg. proof of \cite[Theorem 3.2]{steinberg2016representation}), is included for completeness.    
An alternative proof for inverse categories is given elsewhere \cite[Theorem 2.20]{macedo2012inverse}.  

\begin{proof}
	Suppose (1).  
	Then every $\cat{1}$-morphism has a pseudo-inverse.  
	For each $\cat{1}$-object $o$, $\cat{1}(o,o)$ is an inverse monoid and hence its idempotents commute \cite[Theorem 3.2]{steinberg2016representation}.
	Hence (2).

	Suppose (2).  
	Consider two pseudo-inverses $\beta,\gamma$ to a morphism $\alpha$ in $\cat{1}$.  
	Then $\beta=\beta\alpha\beta=\beta\alpha\gamma\alpha\beta=\beta\alpha\gamma(\alpha\gamma)(\alpha\beta)=\beta\alpha\gamma(\alpha\beta)(\alpha\gamma)=(\beta\alpha)(\gamma\alpha)\beta\alpha\gamma=(\gamma\alpha)(\beta\alpha)\beta\alpha\gamma=\gamma(\alpha\beta\alpha)\beta\alpha\gamma=\gamma\alpha\beta\alpha\gamma=\gamma\alpha\gamma=\gamma$.
	Hence (1).
\end{proof}

\begin{eg}
	Commutative, idempotent monoids are inverse monoids. 
\end{eg}

\begin{lem}
	\label{lem:involution}
	For an inverse category $\cat{1}$, $(-)^{\dagger}$ defines an involution
	$$(-)^{\dagger}:\OP{\cat{1}}\cong\cat{1}.$$
\end{lem}
\begin{proof}
	For each $\cat{1}$-object $o$, $\id_o$ pseudo-inverse to iteslf and hence $\id_o^\dagger=\id_o$ by uniqueness of pseudo-inverses.   

	Consider a composite $\beta\alpha$ of $\cat{1}$-morphisms.  
	First note $\beta\alpha\alpha^\dagger\beta^\dagger\beta\alpha=\beta\beta^\dagger\beta\alpha\alpha^\dagger\alpha=\beta\alpha$, where the second equality follows because idempotent endomorphisms commute in $\cat{1}$ [Lemma \ref{lem:inverse.characterization}].
	Then note $\alpha^\dagger\beta^\dagger\beta\alpha\alpha^\dagger\beta^\dagger=\alpha^\dagger\alpha\alpha^\dagger\beta^\dagger\beta\beta^{\dagger}=\alpha^\dagger\beta^\dagger$, where the second equality follows because idempotent endomorphisms commute in $\cat{1}$ [Lemma \ref{lem:inverse.characterization}].
	Hence $\alpha^\dagger\beta^\dagger$ is pseudo-inverse to $\beta\alpha$.
	Hence $(\beta\alpha)^\dagger=\alpha^\dagger\beta^\dagger$ by uniqueness of pseudo-inverses.  
	
	For each $\cat{1}$-morphism $\zeta$, $\zeta^\dagger$ is pseudo-inverse to $\zeta$ and hence $\zeta=(\zeta^\dagger)^\dagger$ by uniqueness of pseudo-inverses.  
\end{proof}

While not obvious from the definition or even the above characterization, images of inverse categories are inverse.
The proof, almost word-for-word taken from a proof for the inverse monoid case \cite[Theorem 7.6]{clifford1967algebraic}, is included for completeness.

\begin{lem}
	\label{lem:wlog}
	Consider an inverse category $\cat{1}$.
	Each functor
	$$\cat{1}\ra\cat{2}$$
	that is surjective on objects and morphisms has inverse codomain $\cat{2}$.
\end{lem}
\begin{proof}
	Let $F$ denote a functor $\cat{1}\ra\cat{2}$.  

	Each $\cat{2}$-morphism has a pseudo-inverse by $F$ preserving pseudo-inverses and surjective.

	Fix a $\cat{2}$-object $o$.  
	Consider a pair of idempotent endomorphisms in $\cat{2}(o,o)$, of the form $F(\alpha)$ and $F(\beta)$ by $F$ surjective.  
	Note $F(\alpha)F(\beta)=F(\alpha)^2F(\beta)^2=F(\alpha)(\alpha^{\dagger})F(\beta)F(\beta^{\dagger})=F((\alpha\alpha^\dagger)(\beta\beta^\dagger))=F((\beta\beta^\dagger)(\alpha\alpha^{\dagger}))=F(\beta)^2F(\alpha)^2=F(\beta)F(\alpha)$, with the third and sixth equalities by $F(\alpha)=F(\alpha^\dagger)$ and $F(\beta)=F(\beta^\dagger)$ [Lemma \ref{lem:inverse.over.idempotents}].

	Hence $\cat{2}$ is inverse [Lemma \ref{lem:inverse.characterization}].
\end{proof}

There exists an equivalence between the category of small inverse categories and a category of certain double categories \cite{dewolf2015ehresmann}.  
This paper just gives the horizontal groupoid part of this construction.
Fix an inverse category $\cat{1}$.  
Define a groupoid $\gpd{\cat{1}}$ as follows:
\begin{enumerate}
		\item the objects of $\gpd{\cat{1}}$ are all idempotents in submonoids of $\cat{1}$
		\item $\gpd{\cat{1}}(e_1,e_2)$ is the set of all $\cat{1}$-morphisms $\zeta$ for which $\zeta^\dagger\zeta=e_1$ and $\zeta\zeta^\dagger=e_2$
		\item composition in $\gpd{\cat{1}}$ is defined by composition in $\cat{1}$
		\item the inverse of a $\gpd{\cat{1}}$-morphisms $\zeta$ is the pseudo-inverse $\zeta^\dagger$ in $\cat{1}$
\end{enumerate}

Write $[\zeta]$ for the $\gpd{\cat{1}}$-morphism $\zeta^\dagger\zeta\ra\zeta^\dagger\zeta$ defined by a $\cat{1}$-morphism $\zeta$. 

\begin{lem}
	\label{lem:gpd-factorization}
	Fix an inverse category $\cat{1}$.  
  For each factorization $\gamma=\beta\alpha$ in $\cat{1}$, 
	$$[\gamma]=[\beta\alpha\alpha^\dagger][\beta^\dagger\beta\alpha].$$
\end{lem}
\begin{proof}
	Note first that $\cod\,[\beta^\dagger\beta\alpha]=(\beta^\dagger\beta\alpha)(\beta^\dagger\beta\alpha)^\dagger= \beta^\dagger\beta\alpha\alpha^\dagger\beta^\dagger\beta=\beta^\dagger\beta\beta^\dagger\beta\alpha\alpha^\dagger=\beta^\dagger\beta\alpha\alpha^\dagger=\beta^\dagger\beta\alpha\alpha^\dagger\alpha\alpha^\dagger=\alpha\alpha^\dagger\beta^\dagger\beta\alpha\alpha^\dagger=(\beta\alpha\alpha^\dagger)^\dagger\beta\alpha\alpha^\dagger=\dom\,[\beta\alpha\alpha^\dagger]$, with the third and sixth equalities because idempotents commute [Lemma \ref{lem:inverse.characterization}] and the second and seventh equalities because $(-)^\dagger$ is a functorial involution [Lemma \ref{lem:involution}].    
	And $\gamma=\beta\alpha=\beta\beta^\dagger\beta\alpha\alpha^\dagger\alpha=\beta\alpha\alpha^\dagger\beta^\dagger\beta\alpha$, 
	with the last equality because idempotents commute [Lemma \ref{lem:inverse.characterization}].  
\end{proof}

A finite collection of morphisms in an inverse category need not generate a finite inverse subcategory.
The following lemma gives a sufficient condition for when that happens.

\begin{lem}
	\label{lem:quivers}
	Consider the following data.
	\begin{enumerate}
    \item finite directed graph $Q$ having no undirected cycles
		\item inverse category $\cat{2}$
		\item functor $F:Q^*\ra\cat{2}$ from the free category $Q^*$ on $Q$ that is injective on objects
	\end{enumerate}
	Then $F$ factors through a finite inverse category whose submonoids are all idempotent.  
\end{lem}
\begin{proof}
	It suffices to take $F$ to be faithful and surjective on objects without loss of generality [Lemma \ref{lem:wlog}]. 
	Let $\cat{1}$ denote the minimal subcategory of $\cat{2}$ containing the image of $F$ and closed under taking pseudo-inverses. 
	It suffices to show the following:
  \begin{enumerate}
	\item $\cat{1}$ is closed under $(-)^{\dagger}$
	\item The $\cat{1}$-morphisms collectively represent finitely many morphisms in $\gpd{\cat{2}}$.
	\end{enumerate}
	For then $\cat{1}$ would be inverse and finite [Lemma \ref{lem:inverse.characterization}].
	Moreover, each connected subgroupoid of $\gpd{\cat{1}}$ would be freely generated by a tree isomorphic to a subtree of the underlying graph of $Q$, so that $[\zeta]$ is an identity and hence $\zeta=\zeta\zeta^\dagger$ for each $\cat{1}$-endomorphism $\zeta$.

  Fix a $\cat{1}$-morphism $\zeta$. 
	Let $[\zeta]$ denote its associated $\gpd{\cat{1}}$-morphism $\zeta^\dagger\zeta\ra\zeta\zeta^\dagger$.
	There exist directed edges $e_1,e_2,\ldots,e_{2n}$ in $Q$ with $\zeta=F(e_1)F(e_2)^\dagger\cdots F(e_{2n-1})F(e_{2n})^{\dagger}$.  

	Then $\zeta^{\dagger}=F(e_{2n})F(e_{2n-1})^{\dagger}\cdots F(e_1)^{\dagger}$ [Lemma \ref{lem:involution}] and hence (1).

	Let $Q_\zeta$ be the directed subgraph of $Q$ whose edges are $e_1,\ldots,e_{2n}$.
	There exist $\cat{1}$-morphisms $G(e_1),G(e_2),\ldots,G(e_{2n})$ such that $[\zeta]=[G(e_1)][G(e_2)^{\dagger}]\cdots [G(e_{2n})^{\dagger}]$ [Lemma \ref{lem:gpd-factorization}].
	Thus $[G-]$ uniquely extends to a functor $[G]$ from the groupoid-completion of $Q_\zeta^*$ to $\gpd{\cat{1}}$, by universal properties of groupoid-completion. 
	Then
	\begin{align*}
	      [G](e_1e_2^{-1}\cdots e_{2n}^{-1})
			&=[G](e_1)[G](e_2^{-1})\cdots [G](e_{2n}^{-1})\\
    	&=[G](e_1)[G](e_2)^{-1}\cdots [G](e_{2n})^{-1}\\
			&=[G(e_1)][G(e_2)^{\dagger}]\cdots [G(e_{2n})^{\dagger}]\\
	  	&=[\zeta].
  \end{align*}
	Thus (2) because the groupoid-completion of $Q_\zeta^*$ is finite by our assumption on $Q$.  
\end{proof}

\section{Representations}\label{sec:representations}
\textit{Representations} will mean diagrams in $\VECTORSPACES_k$.  
Representations of posets $P$ in the sense of this paper are more general than representations of posets studied elsewhere, which are often taken to mean diagrams $P\ra\VECTORSPACES_k$ sending each morphisms to an inclusion (eg. \cite{scharlau1975subspaces}).  
A \textit{direct sum} of representations of the same category is an object-wise direct sum.  
A representation $\nabla:\smallcat{1}\ra\VECTORSPACES_k$ of a small category $\smallcat{1}$ \textit{factors through a category $\smallcat{2}$} if there exist dotted functors making the following diagram commute:
\begin{equation*}
	\begin{tikzcd}
		& \smallcat{2}\ar[dr,dotted]\\
		\smallcat{1}\ar{rr}[below]{\nabla}\ar[ur,dotted] & & \VECTORSPACES_k 
	\end{tikzcd}
\end{equation*}

We make one general observation about representations of pseudo-inverses.

\begin{lem}
	\label{lem:pseudo-inverse.maps}
	For morphism $\alpha$ with pseudo-inverse $\beta$,
	\begin{align*}
		\im\,\alpha&=\im\,\alpha\beta
		&
		\kernel\,\alpha&=\kernel\,\beta\alpha\\
		\im\,\beta&=\im\,\beta\alpha
		&
		\kernel\,\beta&=\kernel\,\alpha\beta
	\end{align*}
\end{lem}
\begin{proof}
	To prove $\im\,\alpha=\im\,\alpha\beta$, note that the left side is contained in the right side by $\alpha=\alpha\beta\alpha$ and the right side is contained in the left side.
	The equality $\kernel\,\alpha=\kernel\,\beta\alpha$ similarly follows.
	The other equalities follow by switching the roles of $\alpha,\beta$, as $\alpha$ is a pseudo-inverse to $\beta$.  
\end{proof}

We make one non-standard definition.  

\begin{defn}
  An ($\smallcat{1}$-shaped) \textit{blockcode} is a representation
	$$\smallcat{1}\ra\VECTORSPACES_k$$
	sending each $\smallcat{1}$-morphism to a $0$-map or an isomorphism.
\end{defn}

The goal is to investigate and then characterize special properties of blockcodes and their direct sums. 

\subsection{Idempotent monoids}
Representations of idempotent monoids send endomorphisms to projection operators.  
The following characterization for when representations of idempotent monoid factor through commutative, idempotent monoids resembles an existing characterization in the literature for when projections commute \cite[(3.4)]{rehder1980projections}, but without any assumption about the existence of an inner product.  
	
\begin{lem}
	\label{lem:commute}
	Consider an idempotent monoid $M$ and representation 
	$$\nabla:M\ra\VECTORSPACES_k.$$ 
	The following are equivalent:
	\begin{enumerate}
		\item\label{item:factors.through.commutative.monoid} $\nabla$ factors through an idempotent, commutative monoid
		\item\label{item:image.decomposition} For each pair $\alpha,\beta\in M$, $\im\,\alpha=(\im\,\alpha)\cap(\im\,\beta)+(\im\,\alpha)\cap(\kernel\,\beta)$.
		\item\label{item:rank.decomposition} For each pair $\alpha,\beta\in M$, $\dim\,\im\,\alpha=\dim\,(\im\,\alpha)\cap(\im\,\beta)+\dim\,(\im\,\alpha)\cap(\kernel\,\beta)$.
		\item\label{item:image-kernel.decomposition} For each pair $\alpha,\beta\in M$, $S=S\cap(\im\,\beta)+S\cap(\kernel\,\beta)$ for $S=\im\,\alpha,\,\kernel\,\alpha$.
	\end{enumerate}
\end{lem}
\begin{proof}
	Let $\alpha,\beta$ denote morphisms in $M$.
	Take $\nabla$ to be an inclusion without loss of generality, so that $M$ is a monoid of operators on a fixed vector space.

	Suppose (\ref{item:factors.through.commutative.monoid}).  
	Consider $\alpha,\beta$.
	Then $\alpha\beta=\beta\alpha$.  
	Then $\beta(\im\,\alpha)=\beta(\alpha(\im\,\alpha))=\alpha(\beta(\im\,\alpha))\subset\im\,\alpha$. 
	Hence $\im\,\alpha$ is $\beta$-invariant.  
	The only invariant spaces of the projection $\beta$ are direct sums of subspaces of $\kernel\,\beta$ with subspaces in $\im\,\beta$.  
	Hence (\ref{item:image.decomposition}).  

  (\ref{item:image.decomposition}) implies (\ref{item:rank.decomposition}) by taking dimensions.  
	(\ref{item:rank.decomposition}) implies  (\ref{item:image.decomposition}) because in $\im\,\alpha=(\im\,\alpha)\cap(\im\,\beta)+(\im\,\alpha)\cap(\kernel\,\beta)$, the right side is contained in the left side and hence (\ref{item:rank.decomposition}) implies that the right side coincides with the left side.

	Suppose  (\ref{item:image.decomposition}).
	The desired equality in (4) holds for $S=\im\,\alpha$ by assumption.
	The desired equality in (4) holds for $S=\kernel\,\alpha$ by intersecting both sides of $V=(\im\,\alpha)\cap(\im\,\beta)\oplus(\im\,\beta)\cap(\kernel\,\alpha)\oplus\kernel\,\beta$ with $\kernel\,\alpha$.
  Hence (\ref{item:image-kernel.decomposition}).

	Suppose (\ref{item:image-kernel.decomposition}).  
	Consider $\alpha,\beta$.  
	
	For $S=\im\,\alpha,\,\kernel\,\alpha$, $\beta(S)=\beta(S\cap(\im\,\beta))+\beta(S\cap(\kernel\,\beta))\subset S$.
	
	Fix $v$ in the domain/codomain of $\alpha,\beta$.
	Then $\alpha\beta\alpha(v)=\beta\alpha(v)$ by $\beta\alpha(v)\in\,\im\,\alpha$.
	Similarly $\alpha\beta(v-\alpha(v))=0$ by $\beta(v-\alpha(v))\in\kernel\,\alpha$.
	Thus $\alpha\beta(v)=\alpha\beta\alpha(v)+\alpha\beta(v-\alpha(v))=\alpha\beta\alpha(v)=\beta\alpha(v)$.
	
	Hence (\ref{item:factors.through.commutative.monoid}).
\end{proof}

The following observation about commuting projection operators and its proof are elementary but included for completeness.  

\begin{lem}
	\label{lem:commuting.projections}
	Consider a commutative, idempotent monoid $M$ and representaton
	$$\nabla:M\ra\VECTORSPACES_k.$$ 
	For all $\alpha,\beta\in M$, $\im\,\nabla(\beta\alpha)=(\im\,\nabla(\alpha))\cap(\im\,\nabla(\beta))$.  
\end{lem}
\begin{proof}
  Take $\nabla$ to be inclusion without loss of generality.
  The equality $\im\,\alpha\beta=(\im\,\alpha)\cap(\im\,\beta)$ follows because: the left side is contained in the right side because $\im\,\alpha\beta\subset\im\,\alpha$ and $\im\,\beta\alpha\subset\im\,\beta$; and the right side is contained in the left side because $\alpha\beta(\im\,\alpha)\cap(\im\,\beta))=\alpha(\im\,\alpha)\cap(\im\,\beta)=(\im\,\alpha)\cap(\im\,\beta)$.  
\end{proof}

We then show that every representation of a commutative, idempotent monoid lifts to a category of inner product spaces and partial isometries.

\begin{lem}
	\label{lem:unitary.projections}
	Every representation of a commutative idempotent monoid
	$$\nabla:M\ra\VECTORSPACES_{\C}$$
	lifts to a monoid of unitary projections on a fixed Hermitian vector space.  
\end{lem}
\begin{proof}
	Take $\nabla$ to be inclusion without loss of generality, so that $M$ is a monoid of operators on a fixed vector space denoted by $\mathfrak{v}$.  

	Let $\alpha,\beta,\gamma$ denote $M$-morphisms.
	
	Further assume all operators of the form $\id_{\mathfrak{v}}-\gamma$ are $M$-morphisms because otherwise we can replace $M$ with the submonoid of $\VECTORSPACES_{\C}$ generated by all operators of the form $\gamma$ or $\id_{\mathfrak{v}}-\gamma$; such a monoid is commutative because it has mutually commuting generators: $(\id_{\mathfrak{v}}-\alpha)\beta=\beta-\alpha\beta=\beta-\beta\alpha=\beta(\id_{\mathfrak{v}}-\alpha)$, $(\id_{\mathfrak{v}}-\beta)(\id_{\mathfrak{v}}-\alpha)=\id_{\mathfrak{v}}-\alpha-\beta-\beta\alpha=\id_{\mathfrak{v}}-\beta-\alpha-\alpha\beta=(\id_{\mathfrak{v}}-\alpha)(\id_{\mathfrak{v}}-\beta)$ for all $\alpha,\beta$; such a commutative monoid is  idempotent because it has idempotent generators: $\gamma^2=\gamma$ by assumption and $(\id_{\mathfrak{v}}-\gamma)^2=\id_{\mathfrak{v}}-2\gamma+\gamma=\id_{\mathfrak{v}}-\gamma$ for all $\gamma$.  
	
	Let $\mathfrak{D}$ be the poset of all decompositions of $\mathfrak{v}$ into direct sums of images of $M$-morphisms ordered by refinement of decompositions.
	The poset $\mathfrak{D}$ is non-empty because it contains the minimum decomposition $\mathfrak{v}=\im\,\id_{\mathfrak{v}}$.  
	Every chain in $\mathfrak{D}$ is finite by $\mathfrak{v}$ finite-dimensional.
	Thus $\mathfrak{D}$ has a maximal such decomposition of the form $\mathfrak{v}=\oplus_{\phi\in\mathcal{I}}\im\,\phi$.
	Equip $\mathfrak{v}$ with a Hermitian inner product so that $(\im\,\alpha)\perp(\im\,\beta)$ for distinct $\alpha,\beta\in\mathcal{I}$.
	For each $M$-morphism $\zeta$, maximality implies the two decompositions
	$$\bigoplus_{\phi\in\mathcal{I}}\im\,\zeta\phi\oplus\im\,(\id_{\mathfrak{v}}-\zeta)\phi=\bigoplus_{\phi\in\mathcal{I}}(\im\,\phi)$$
	of $\mathfrak{v}$ are the same up to reordering of the summands so that $\im\,\zeta\perp\im\,(\id_{\mathfrak{v}}-\zeta)=\kernel\,\zeta$ and hence $\zeta$ is unitary.
\end{proof}

\subsection{Semilattices}
This section relates representations of unital semilattices in the sense of posets with maxima and all finite infima with representations of unital semilattices in the sense of commutative idempotent monoids.
For each representation 
$$\nabla:M\ra\VECTORSPACES_k,$$
of a commutative idempotent monoid, let $\im\,\nabla$ denote the representation $\Omega_M\ra\VECTORSPACES_k$ of the associated poset $\Omega_M$, whose objects are the morphisms of $M$ with $\alpha\leqslant_{\Omega_P}\beta$ if $\alpha=\alpha\beta$, sending each $\Omega_M$-object $\alpha$ to $\im\,\alpha$ and each relation $\alpha\leqslant_{\Omega_P}\beta$ to an inclusion $\im\,\alpha\ira\im\,\beta$ [Lemma \ref{lem:commuting.projections}]. 
The following lemma chararacterizes when a suitable representation of such a poset arises as the image of the representation of a commutative, idempotent monoid.

\begin{lem}
  \label{lem:realizability}
  Fix a finite, complete poset $P$ with a maximum and a representation
  $$\nabla:P\ra\VECTORSPACES_k$$
  that sends every morphism to an inclusion.
  The following are equivalent:
  \begin{enumerate}
    \item For all $P$-objects $\mathfrak{b},\mathfrak{c}$, $\sum_{\mathfrak{a}\leqslant_P\mathfrak{b}}\mu_P(\mathfrak{a})(\dim\mathfrak{a}-\dim(\mathfrak{a}\cap\mathfrak{c}))\geqslant 0$.  
    \item There exists a representation $\Pi:\Sigma_P\ra\VECTORSPACES_k$ such that $\im\,\Pi=\nabla$, where $\Sigma_P$ is the commutative, idempotent monoid whose morphisms are the objects of $P$ and whose multiplication is defined by binary infima.    
  \end{enumerate}
\end{lem}
\begin{proof}
  Take $\nabla$ to be inclusion without loss of generality.
  Let $\mathfrak{a},\mathfrak{b},\mathfrak{c}$ denote $P$-objects.
  Let 
  $$\phi_{\mathfrak{c}}(\mathfrak{a})=(\dim\,\mathfrak{a}-\dim\,\mathfrak{a}\cap\mathfrak{c})-\dim\,\sum_{\mathfrak{b}<\mathfrak{a}}\mathfrak{b}+\dim\,\sum_{\mathfrak{b}<\mathfrak{a}}\mathfrak{b}\cap\mathfrak{c}$$
	
  Note $\dim\,\mathfrak{b}-\dim\,\mathfrak{b}\cap\mathfrak{c}=\sum_{\mathfrak{a}\leqslant_P\mathfrak{b}}\phi_{\mathfrak{c}}(\mathfrak{a})$.
  M\"{o}bius Inversion implies 
  $$\phi_{\mathfrak{c}}(\mathfrak{b})=\sum_{\mathfrak{a}\leqslant_P\mathfrak{b}}\mu_P(\mathfrak{a})(\dim\mathfrak{a}-\dim(\mathfrak{a}\cap\mathfrak{c})).$$ 
  \vspace{.1in}\\
  \textit{(1)$\implies$(2):}
  Assume (1).  
  For each $\mathfrak{b}$, there exists a choice $\Delta_{\mathfrak{b},\mathfrak{c}}$ of $\phi_{\mathfrak{c}}(\mathfrak{b})$ linearly independent vectors in the set-theoretic difference $\mathfrak{b}-\mathfrak{c}-\sum_{\mathfrak{a}<_P\mathfrak{b}}\mathfrak{a}$ by our assumption (1) that $\phi_{\mathfrak{c}}(\mathfrak{b})\geqslant 0$.
  For each $\mathfrak{c}$, let $\Pi(\mathfrak{c})$ be the unique projection operator with image $\mathfrak{c}$ and kernel spanned by the linearly independent set $\amalg_{\mathfrak{b}}\Delta_{\mathfrak{b},\mathfrak{c}}$.  
  For all $\mathfrak{b},\mathfrak{c}$,
  \begin{align*}
	     \dim\,(\im\,\Pi(\mathfrak{b})\cap\kernel\,\Pi(\mathfrak{c}))
			 &= \dim\,(\mathfrak{b}\cap\langle \amalg_{\mathfrak{a}}\Delta_{\mathfrak{a},\mathfrak{c}}\rangle)\\
			 &= \#(\amalg_{\mathfrak{a}\leqslant\mathfrak{b}}\Delta_{\mathfrak{a},\mathfrak{c}})\\
			 &= \sum_{\mathfrak{a}\leqslant\mathfrak{b}}\phi_{\mathfrak{c}}(\mathfrak{a})\\
			 &= \dim\,\mathfrak{b}-\dim\,\mathfrak{b}\cap\mathfrak{c}\\
			 &= \dim\,\im\,\Pi(\mathfrak{b})-\dim\,\im\,\Pi(\mathfrak{b})\cap\im\,\Pi(\mathfrak{c})
	\end{align*}
	and hence $\Pi(\mathfrak{b})\Pi(\mathfrak{c})=\Pi(\mathfrak{c})\Pi(\mathfrak{b})$ [Lemma \ref{lem:commute}].
	Thus (2).
	\vspace{.1in}\\
	\textit{(2)$\implies$(1):}
	Conversely, suppose (2).
	The smallest poset $P^+$ containing $P$ as a sub-poset and all finite sums of $P$-objects is also finitely complete.
	And $\Pi$ extends to a representation $\Pi^+$ of $\Sigma_{P^+}$ by recursively defining $\Pi^+(\mathfrak{d}+\mathfrak{e})=\Pi^+(\mathfrak{d})+\Pi^+(\mathfrak{e})-\Pi^+(\mathfrak{d}\cap\mathfrak{e})$ and noting that linear combinations of commuting operators commute.
	For all $P^+$-objects $\mathfrak{d},\mathfrak{e}$, 
	\begin{align*}
		  \dim\,\mathfrak{d}
		&=\dim\im\,\Pi^+({\mathfrak{d}})\\
		&=\dim(\im\,\Pi^+({\mathfrak{d}})\cap\im\,\Pi^+({\mathfrak{e}}))+\dim(\im\,\Pi^+({\mathfrak{d}})\cap\kernel\,\Pi^+({\mathfrak{e}})) && [\mathrm{Lemma}\;\ref{lem:commute}]\\
		&=\dim\,\mathfrak{d}\cap\mathfrak{e}+\dim(\mathfrak{d}\cap\kernel\,\Pi^+({\mathfrak{e}}))
	\end{align*}
	and hence for all $\mathfrak{b},\mathfrak{c}$, let $\mathfrak{d}=\sum_{\mathfrak{a}<\mathfrak{b}}\mathfrak{a}$ and note
	\begin{align*}
		0&\geqslant \dim(\mathfrak{b}\cap\kernel\,\Pi(\mathfrak{c}))-\dim\,\mathfrak{d}\cap\kernel\,\Pi(\mathfrak{c})\\
		&=(\dim\,\mathfrak{b}-\dim\,\mathfrak{b}\cap\mathfrak{c})-\dim\,\mathfrak{d}+\dim\,\mathfrak{d}\cap\mathfrak{c}	
		=\phi_{\mathfrak{c}}(\mathfrak{b})
	  =\sum_{\mathfrak{a}\leqslant_P\mathfrak{b}}\mu_P(\mathfrak{a})(\dim\mathfrak{a}-\dim(\mathfrak{a}\cap\mathfrak{c}))
  \end{align*}
	and hence (1).  
\end{proof}

\subsection{Inverse categories}
Representations of inverse categories have special properties.
For example, complex representations of inverse categories are partial isometries up to isomorphisms in a certain sense [Propoition \ref{prop:inverse-category-representation}]. 
We conclude this section by recalling and specializing an extension of Maschke's Theorem to the inverse setting \cite[Theorem 4.1]{linckelmann2013inverse}. 

\begin{lem}
	\label{lem:kernel.decomposition}
	For morphisms $\alpha,\beta$ in an inverse subcategory of $\VECTORSPACES_k$ with $\cod\,\alpha=\cod\,\beta$,
	$$\beta^{-1}(\im\,\alpha)=\kernel\,\beta+\im\,\beta^\dagger\alpha.$$
\end{lem}
\begin{proof}
	Take $v\in\beta^{-1}(\im\,\alpha)$.
	There exists $a\in\dom(\alpha)$ such that $\beta(v)=\alpha(a)$
	Then $v=(\id-\beta^{\dagger}\beta)(v)+\beta^{\dagger}\beta(v)=(\id-\beta^{\dagger}\beta)(v)+\beta^{\dagger}\alpha(a)$.
	Note that $(\id-\beta^{\dagger}\beta)(v)\in\kernel\,\beta^\dagger\beta=\kernel\,\beta$ [Lemma \ref{lem:pseudo-inverse.maps}] and $\beta^\dagger\alpha(a)\in\im\,\beta^\dagger\alpha$.  
	Thus the left side is contained in the right side.  

	Note that $\beta(\kernel\,\beta)={\bf 0}\subset\im\,\alpha$.  
	And $\im\,\beta(\beta^\dagger\alpha)=\im\,(\beta\beta^\dagger)(\alpha\alpha^\dagger)=(\alpha\alpha^\dagger)(\beta\beta^\dagger)\in\im\,\alpha$ [Lemma \ref{lem:pseudo-inverse.maps}, \ref{lem:inverse.characterization}].
	Thus the right side is contained in the left side.
\end{proof} 

Every adjoint-closed category of inner product spaces and partial isometries is inverse.  
However, not every complex representation of a small category that factors through an inverse category lifts to a category of inner product spaces and partial isometries.

\begin{eg}
	\label{eg:non-isometric.groupoid}
	Consider the following faithful diagram $\nabla$:
	\begin{equation*}
			\begin{tikzcd}
				& \C\ar[out=480,in=600,loop,swap,"{z\mapsto 2z}"]
			\end{tikzcd}
	\end{equation*}
	Then $\nabla$ factors through a groupoid and hence an inverse category, but does not lift to a category of inner product spaces and partial isometries.  
\end{eg}

The following proposition shows the next best thing: that every complex representation of a small inverse category lifts to a category of inner product spaces and partial isometries up to isomorphism in the following sense.
Fix a Hermitian vector space $\mathfrak{c}$ with Hermitian inner product $q$.  
For each vector subspace $\mathfrak{b}\leqslant\mathfrak{c}$, $\mathfrak{b}^{\perp_q}$ will denote the orthgonal complement of $\mathfrak{b}$ in $\mathfrak{c}$ with respect to $q$.  
A \textit{$q$-unitary projection} is a projection $\phi:\mathfrak{c}\ra\mathfrak{c}$ with $\kernel\,\phi\perp\im\,\phi$ with respect to $q$; $q$-unitary projections on $\mathfrak{c}$ are exactly the idempotent partial isometries on $\mathfrak{c}$.

\begin{prop}
	\label{prop:inverse-category-representation}
	Consider an inverse category $\smallcat{1}$ and functor
	$$\nabla:\smallcat{1}\ra\VECTORSPACES_\C.$$
	There exists a choice of Hermitian inner product $q(o)$ on $\nabla(o)$ for each $\smallcat{1}$-object $o$ such that for each $\smallcat{1}$-morphism $\zeta:x\ra y$, $\nabla(\zeta\zeta^\dagger)$ is $q(x)$-unitary projection onto $(\kernel\,\nabla(\zeta))^{\perp_{q(x)}}$ and $\nabla(\zeta^\dagger\zeta)$ is $q(y)$-unitary projection onto $\im\,\nabla(\zeta)$.
\end{prop}
\begin{proof}
	Take $\nabla$ to be inclusion without loss of generality [Lemma \ref{lem:wlog}].  
	
	Consider an $\smallcat{1}$-object $\mathfrak{v}$. 
	Let $\Lambda_{\mathfrak{v}}$ denote the submonoid of the inverse endomorphism monoid of all linear operators on $\mathfrak{v}$ consisting of the projections.
	Then $\Lambda_{\mathfrak{v}}$ is a commutative, idempotent monoid because it is an idempotent inverse monoid [Lemma \ref{lem:inverse.characterization}].  
	Hence there exists an inner product $q(\mathfrak{v})$ on $\mathfrak{v}$ such that every linear operator on $\mathfrak{v}$ of the form $\psi^\dagger\psi$ is a $q(\mathfrak{v})$-unitary projection [Lemma \ref{lem:unitary.projections}].  
	
	Consider an $\smallcat{1}$-morphism $\psi$.
	Then $\im\,\psi=\im\,\psi\psi^{\dagger}$ [Lemma \ref{lem:pseudo-inverse.maps}].
	Then $\psi\psi^{\dagger}$ is $q(\cod\,\psi)$-unitary projection onto $\im\,\psi$.
	Similarly, $\psi^{\dagger}\psi$ is the unique $q(\dom\,\psi)$-unitary projection onto $\im\,\psi^\dagger$.
	Moreover, $\kernel\,\psi=\kernel\,\psi^{\dagger}\psi$ [Lemma \ref{lem:pseudo-inverse.maps}] is orthogonal to $\im\,\psi^{\dagger}=\im\,\psi^\dagger\psi$[Lemma \ref{lem:pseudo-inverse.maps}] with respect to $q(\dom\,\psi)$ by $\psi^\dagger\psi$ $q(\dom\,\psi)$-unitary.   
	Thus $\im\,\psi^{\dagger}=(\kernel\,\psi)^{\perp_{q(\dom\,\psi)}}$.  
\end{proof}

Every faithful representation of an inverse monoid as partial isometries extends to a faithful representation of an inverse monoid as partial isometries whose unitary projections form Boolean lattices \cite[Theorem 2.5]{bracci2007representations}.
The following lemma is an analogue, where inverse monoids are generalized to inverse categories and no metric structure is assumed.

\begin{lem}
	\label{lem:completion}
	Consider an inverse category $\smallcat{1}$ that is a subcategory 
	$$\smallcat{1}\subset\VECTORSPACES_k.$$
	Let $\smallcat{2}$ be the minimal subcategory of $\VECTORSPACES_k$ containing $\smallcat{1}$ and, for each $\smallcat{1}$-objects $V$, all idempotent $k$-linear combinations of idempotent operators on $V$ in $\smallcat{1}$.
	Then $\smallcat{2}$ is inverse.
\end{lem}
\begin{proof}
	Fix an $\smallcat{1}$-object $\mathfrak{v}$.  
	Let $\Lambda_{\mathfrak{v}}$ denote the set of all idempotent $k$-linear combinations of projection operators on an $\smallcat{1}$-object $\mathfrak{v}$.
	The elements in $\Lambda_{\mathfrak{v}}$ commute because idempotents in $\smallcat{1}(\mathfrak{v},\mathfrak{v})$ commute [Lemma \ref{lem:inverse.characterization}].

	Let $\zeta$ denote an $\smallcat{1}$-morphism. 
	
	For each $\zeta$ and linear operator $\tau$ on the domain of $\psi$, let $\conjugate{\tau}{\zeta}=\zeta\tau\zeta^\dagger$.
  
	Consider $\zeta$ and $\tau\in\Lambda_{\mathfrak{dom}\,\zeta}$.  
	Then $\zeta^\dagger\zeta$ commutes with all projection operators on $\mathfrak{v}$ in $\smallcat{1}$ [Lemma \ref{lem:inverse.characterization}] and hence with $\pi$.
	In particular $(\conjugate{\tau}{\zeta})^2=\zeta\tau(\zeta^\dagger\zeta)\tau\zeta^\dagger=\zeta\zeta^\dagger\zeta\tau^2\zeta^\dagger=\zeta\tau\zeta^\dagger=\conjugate{\tau}{\zeta}$ and $\conjugate{\tau}{\zeta}$ is $k$-linear in $\tau$.
	Hence $\conjugate{\tau}{\zeta}\in\Lambda_{\mathfrak{cod}\,\zeta}$
	And moreover $\zeta\tau=\zeta(\zeta^\dagger\zeta)\tau=\zeta\tau(\zeta^\dagger\zeta)=\conjugate{\tau}{\zeta}\zeta$.

	Let $(\sigma,\zeta)$ denote a pair with $\zeta$ an $\smallcat{1}$-morphism and $\sigma\in\Lambda_{\mathfrak{cod}\,\zeta}$.
	It follows inductively that all linear maps of the form $\sigma\zeta$ for all possible choices $(\sigma,\zeta)$ are closed under	composition and hence comprise all the $\smallcat{2}$-morphisms by minimality.  
	It then follows that each $\smallcat{2}$-morphism, of the form $\sigma\zeta$ for a choice $(\sigma,\zeta)$, admits a pseudo-inverse $\zeta^{\dagger}\sigma$ because $(\sigma\zeta)(\zeta^{\dagger}\sigma)(\sigma\zeta)=\sigma(\zeta\zeta^{\dagger})\sigma\sigma\zeta=(\zeta\zeta^{\dagger})\sigma\zeta=\sigma\zeta\zeta^\dagger\zeta=\sigma\zeta$ and $(\zeta^{\dagger}\sigma)(\sigma\zeta)(\zeta^{\dagger}\sigma)=\zeta^{\dagger}\sigma(\zeta(\zeta^{\dagger})\sigma=\zeta^\dagger(\zeta(\zeta^{\dagger})\sigma=\zeta^\dagger\sigma$.  
	Every idempotent in $\smallcat{2}$, of the form $\sigma\zeta$ for a choice $(\sigma,\zeta)$, satisfies $\sigma\zeta=\sigma\zeta\zeta^\dagger$ and therefore a composite of $\Lambda_{\mathfrak{dom}\,\zeta}$-operators, lies in $\Lambda_{\mathfrak{dom}\,\zeta}$.
	Thus the idempotent endomorphisms in $\smallcat{2}$ commute.
	Hence $\smallcat{2}$ is inverse [Lemma \ref{lem:inverse.characterization}].
\end{proof}

Fix \textit{finite} inverse category $\smallcat{1}$ and ring $R$.
There exists an $R$-algebra isomorphism
\begin{equation}
	\label{eqn:decomposition}
  R[\smallcat{1}]\cong R[\gpd{\smallcat{1}}]
\end{equation}
from the associated category $R$-algebra $R[\smallcat{1}]$ to the groupoid $R$-algebra $R[\gpd{\smallcat{1}}]$, sending each $\smallcat{1}$-morphism $\zeta$ to an $R$-linear sum of $\gpd{\smallcat{1}}$-morphisms of the form $[\zeta\pi]$ with $\pi$ an idempotent endomorphism \cite[Theorem 4.1]{linckelmann2013inverse}.
Specialize now to the case where each submonoid of $\smallcat{1}$ is idempotent.
For each $\smallcat{1}$-endomorphism $\zeta$, $\zeta=\zeta^\dagger=\zeta\zeta^\dagger=\zeta^\dagger\zeta$ by uniqueness of pseudo-inverses  and hence $[\zeta]$ is an identity.
Thus the aforementioned isomorphism (\ref{eqn:decomposition}) sends each $\smallcat{1}$-morphism $\zeta\in\smallcat{1}(x,y)$ to an $R$-linear sum of identities or an $R$-linear sum of isomorphisms whose preimages under (\ref{eqn:decomposition}) are isomorphisms between distinct objects.  
We further specialize to the case $R=k$ to conclude the following observation.

\begin{prop}
	\label{prop:finite.inverse}
	Consider the following data.
  \begin{enumerate}
		\item finite inverse category $\smallcat{1}$ whose submonoids are all idempotent
		\item representation $\nabla:\smallcat{1}\ra\VECTORSPACES_k$
	\end{enumerate}
	Then $\nabla$ is a direct sum of blockcodes.
\end{prop}

\subsection{Categories}
This section investigates the property of factorizability of a general representation through an inverse category.
This kind of factorizability is stable under direct sums.

\begin{lem}
	\label{lem:direct.sums}
  Consider a sequence of representations
	$$\nabla_1,\nabla_2,\ldots,\nabla_n:\smallcat{1}\ra\VECTORSPACES_k$$
	that each factor through (possibly different) inverse categories.  
	Then $\bigoplus_i\nabla_i$ factors through an inverse category.
\end{lem}
\begin{proof}
	Suppose for each $i$, $\nabla_i=\nabla'_i\tau_i$ for functors $\tau_i:\smallcat{1}\ra\smallcat{2}_i$ and $\nabla'_i:\smallcat{2}_i\ra\VECTORSPACES_k$ with $\smallcat{2}_i$ inverse.
	Let $\smallcat{2}=\prod_i\smallcat{2}_i$.  
	Then $\smallcat{2}$ is inverse [Lemma \ref{lem:inverse.products}].  
	Let $\nabla'$ be the representation $\smallcat{2}\ra\VECTORSPACES_k$ naturally sending an object $(y_1,y_2.,\ldots,y_n)$ to $\bigoplus_{i}\nabla'_i(y_i)$.  
  Then $\bigoplus_i\nabla_i=\nabla'(\prod_i\tau_i)$.
\end{proof}

This kind of factorizability generalizes direct sums of blockcodes.

\begin{prop}
  \label{prop:one-way}
  Direct sums of blockcodes factor through inverse categories.
\end{prop}
\begin{proof}
  Consider a blockcode $\nabla:\smallcat{1}\ra\VECTORSPACES_k$.
  It suffices to show $\nabla$ factors through an inverse category [Lemma \ref{lem:direct.sums}]. 
  It further suffices to take $\nabla$ to be an inclusion [Lemma \ref{lem:wlog}].

  Consider an $\smallcat{1}$-morphism $\psi$.  
  Either $\psi$ is a $0$-map or an isomorphism.
  In the former case, take $\psi^{\dagger}$ to be the $0$-map $\cod(\psi)\ra\dom(\psi)$.  
  In the latter case, take $\psi^{\dagger}=\psi^{-1}$.  
	
  Let $\smallcat{2}$ be the category generated by all maps of the form $\psi$ and $\psi^{\dagger}$ for $\psi$ an $\smallcat{1}$-morphism.  
  Every $\smallcat{2}$-morphism is then either a $0$-map or a composite of isomorphisms in $\smallcat{2}$ whose inverses are in $\smallcat{2}$.
  In either case, each $\smallcat{2}$-morphism admits a pseudo-inverse in $\smallcat{2}$.
  Moreover, the idempotent endomorphisms in $\smallcat{2}$, $0$-maps and identities, commute.  
  Thus $\smallcat{2}$ is inverse [Lemma \ref{lem:inverse.characterization}].
\end{proof}

The salient structure needed to determine factorizability of a representation through an inverse category is the data of all associated kernels, images, and intersections thereof in a diagram.

\begin{defn}
	\label{defn:multiflag}
	Fix a representation of the form
	$$\nabla:\smallcat{1}\ra\VECTORSPACES_k.$$
	Let $\flag_\nabla$ denote the functor that is initial among all functors $F:\smallcat{1}\ra\CATS$ sending each object $o$ to the poset of linear subspaces of $\nabla(o)$ ordered by inclusion, closed under finite intersection, containing both ${\bf 0}$ and $\nabla(o)$ such that for each $\smallcat{1}$-morphism $\zeta$: (1) $F(\zeta)(\mathfrak{a})=\nabla(\zeta)(\mathfrak{a})$ for all $\mathfrak{a}\in F(\dom\,\zeta)$; and (2) $\nabla(\zeta)^{-1}\mathfrak{b}\in F(\dom\,\zeta)$ for all $\mathfrak{b}\in F(\cod\,\zeta)$.
\end{defn}

The well-posedness of the definition follows because the class of functors $F:\smallcat{1}\ra\CATS$ satisfying the above conditions (1)-(3) are closed under object-wise intersections and contains the functor sending each object $o$ to the poset of all linear subspaces of $\nabla(o)$.  
Concretely, $\flag_{\nabla}$ can be iteratively constructed from $\nabla$ by starting with the minimal function from the objects of $\smallcat{1}$ to the objects of $\CATS$ satisfying (1) and iteratively adding objects to make (2) and (3) hold.  
For certain representations $\nabla$ of inverse categories, $\flag_{\nabla}$ is just the data of all possible images of maps in the diagram $\nabla$.    

\begin{lem}
	\label{lem:inverse.flags}
	Consider an inverse category $\smallcat{1}$ that is a subcategory 
	$$\smallcat{1}\subset\VECTORSPACES_k$$
	such that all $k$-linear combinations of projections operators in $\smallcat{1}$ on a fixed $\smallcat{1}$-object that are themselves projections are in $\smallcat{1}$.
	The underlying set of $\flag_{\smallcat{1}\ira\VECTORSPACES_k}(\mathfrak{v})$ consists of all images of all $\smallcat{1}$-morphisms having codomain $\mathfrak{v}$.
\end{lem}
\begin{proof}
	Let $\mathfrak{a},\mathfrak{b}$ denote $\smallcat{1}$-objects and $\zeta$ denote an $\smallcat{1}$-morphism.  
	Let
	$$F(\mathfrak{b})=\!\!\!\!\!\bigcup_{\cod(\zeta)=\mathfrak{b}}\!\!\!\!\!\{\im\,\zeta\}$$
	Then $F(\mathfrak{b})\subset\flag_{\smallcat{1}\ira\VECTORSPACES_k}(\mathfrak{b})$ by construction and ${\bf 0}=\im\,(\id_{\mathfrak{b}}-\id_{\mathfrak{b}}),\mathfrak{b}=\im\,\id_{\mathfrak{b}}\in F(\mathfrak{b})$.
	It therefore suffices to show $F$ satisfies (1)-(2) of Definition \ref{defn:multiflag}, and in particular also defines a functor $\smallcat{1}\ra\CATS$.  
	The result then follows from the initality of $\flag_{\smallcat{1}\ira\VECTORSPACES_k}$.

	Fix an $\smallcat{1}$-morphism $\tau$.

	Consider $\zeta$ with $\cod\,\zeta=\dom\,\tau$.  
	Then $\tau(\im\,\zeta)=\im\,\tau\zeta\in\flag_{\smallcat{1}\ira\VECTORSPACES_k}(\cod\,\tau)$.

	Consider $\zeta$ with $\cod\,\zeta=\cod\,\tau$.
	Let $\sigma=\tau^\dagger\zeta\zeta^\dagger$.
	Let $\pi_1=\id_{\cod\,\tau}-\tau^\dagger\tau$ and $\pi_2=\tau^\dagger\zeta\zeta^\dagger\tau$.
	Then $\pi_1,\pi_2,\pi_1\pi_2$ and hence also $\pi_1+\pi_2-\pi_1\pi_2$ are all idempotent endomorphisms in $\smallcat{1}$ [Lemma \ref{lem:commute}]. 
	Note $\tau^{-1}(\im\,\zeta)=\kernel\,\tau+\im\,\tau^\dagger\zeta=\kernel\,\tau^\dagger\tau+\im\,\pi_2=\im\,\pi_1+\im\,\pi_2=\im\,(\pi_1+\pi_2-\pi_1\pi_2)\in F(\cod\,\tau)$ [Lemma \ref{lem:kernel.decomposition}].
\end{proof}

\begin{thm}
	\label{thm:inverse-extendability}
	\InverseExtendability{}
\end{thm}
\begin{proof}
	Take $\nabla$ to be inclusion without loss of generality [Lemma \ref{lem:wlog}]. 
	Let $\zeta'$ denote the corestriction of a linear map $\zeta$ to its image.
	\vspace{.1in}\\
	\textit{(1)$\implies$(2):}
	Assume (\ref{item:factors}).  
	There exists an inverse subcategory $\smallcat{2}$ of $\VECTORSPACES_k$ containing $\smallcat{1}$ [Lemma \ref{lem:wlog}].   
	We can assume all idempotent $k$-linear combinations of idempotents in $\smallcat{2}$ are in $\smallcat{2}$ without loss of generality [Lemma \ref{lem:completion}].  
	Consider $\smallcat{1}$-object $o$.
	Then $\flag_{\nabla}(o)\subset\flag_{\smallcat{2}\ira\VECTORSPACES_k}(o)$.  
	Hence each $\mathfrak{c}\in\flag_{\nabla}(o)$ is the image of a linear map $\pi_{\mathfrak{c}}$ in $\smallcat{2}$ [Lemma \ref{lem:inverse.flags}], which we can take to be projection [Lemma \ref{lem:pseudo-inverse.maps}]. 
	And $\pi_{\mathfrak{b}}\pi_{\mathfrak{c}}=\pi_{\mathfrak{c}}\pi_{\mathfrak{b}}$ for all $\mathfrak{b},\mathfrak{c}\in\flag_{\nabla}(o)$ [Lemma \ref{lem:inverse.characterization}] .
	Hence (2) [Lemma \ref{lem:commute}].
	\vspace{.1in}\\
	\textit{(2)$\implies$(1):}
	Assume (2).  
	For each $\smallcat{1}$-object $\mathfrak{v}$ and $\mathfrak{c}\in\flag_{\nabla}(\mathfrak{v})$, there exists a projection $\pi_{\mathfrak{c}}$ on $\mathfrak{v}$ with image $\mathfrak{c}$ such that $\pi_{\mathfrak{a}}\pi_{\mathfrak{b}}=\pi_{\mathfrak{b}}\pi_{\mathfrak{a}}$ for all $\mathfrak{a},\mathfrak{b}\in\flag_{\nabla}(\mathfrak{v})$ [Lemma \ref{lem:commute}].  

	Fix an $\smallcat{1}$-object $\mathfrak{v}$. 
	Let $M_{\mathfrak{v}}$ denote the set of all idempotent $k$-linear combinations of projections of the form $\pi_{\mathfrak{c}}$ for $\mathfrak{c}\in\flag_{\nabla}(\mathfrak{v})$.
	The operators in $M_{\mathfrak{v}}$ commute because $k$-linear combinations of commuting operators commute.
	The operators in $M_{\mathfrak{v}}$ are closed under composition [Lemma \ref{lem:commute}].  
	Hence $M_{\mathfrak{v}}$ can be regarded henceforth as a commutative, idempotent monoid.  
	Let $\smallcat{2}$ denote the class of all linear maps $\zeta$ between $\smallcat{1}$-objects with $\kernel\,\zeta$ the kernel of an operator $\pi'\in M_{\dom\,\psi}$ and $\im\,\zeta$ the image of an operator $\pi''\in M_{\cod\,\psi}$ satisfying $\zeta=\pi''\zeta\pi'$.  

	Fix a $\smallcat{2}$-morphism $\zeta$.  
	There exists a dotted isomorphism making the left of the diagrams
	\begin{equation*}
    \begin{tikzcd}
			\dom\,\zeta\ar{d}[description]{\pi'_{\im\,\zeta^*}=(\id_{\dom\,\zeta}-\pi_{\kernel\,\zeta})'}\ar{rr}[above]{\zeta} 
			& & \cod\,\zeta
			\\
			\im\,(\id_{\dom\,\zeta}-\pi_{\kernel\ \zeta})\ar[dotted]{rr}[below]{\zeta_{\mathfrak{iso}}}
			& & \im\,\zeta\ar[u,hookrightarrow]
		\end{tikzcd}
    \begin{tikzcd}
		  	\cod\,\zeta\ar{d}[left]{\pi'_{\im\,\zeta}}\ar{r}[above]{\zeta^*} 
			& \dom\,\zeta
			\\
		  	\im\,\zeta\ar{r}[below]{\zeta_{\mathfrak{iso}}^{-1}}
			& \im\,(\id_{\dom\,\zeta}-\pi_{\kernel\,\zeta})\ar[u,hookrightarrow]
		\end{tikzcd}
	\end{equation*}
	commute.  
	The map $\zeta^*$ making the right diagram commute is in $\smallcat{2}$ by $\kernel\,\zeta^*=\kernel\,\pi'_{\im\,\zeta}=\kernel\,\pi_{\im\,\zeta}=\im\,\id_{\cod\,\zeta}-\pi_{\im\,\zeta}$ and $\im\,\zeta^*=\im\,(\id_{\cod\,\zeta}-\pi_{\kernel\,\zeta})$.
  Note
  \begin{align*}
		\zeta^*\zeta&=
		(\im\,(\id_{\dom\,\zeta}-\pi_{\kernel\,\zeta})\ira\dom\,\zeta)\zeta_{\mathfrak{iso}}^{-1}\zeta_{\mathfrak{iso}}(\id_{\dom\,\zeta}-\pi_{\kernel\,\zeta})'\\
		&=(\im\,(\id_{\dom\,\zeta}-\pi_{\kernel\,\zeta})\ira\dom\,\zeta)(\id_{\dom\,\zeta}-\pi_{\kernel\,\zeta})'\\
		&=\id_{\dom\,\zeta}-\pi_{\kernel\,\zeta}
	\end{align*}
	and similarly $\zeta\zeta^*=\pi_{\im\,\zeta}$.
	Thus we have that $\zeta\zeta^*\zeta=\zeta(\id_{\dom\,\zeta}-\pi_{\kernel\,\zeta})=\zeta$ and $\zeta^*\zeta\zeta^*=\pi_{\im\,\zeta}\zeta^*=\zeta^*$.  
	
	Suppose $\zeta$ is an idempotent in a submonoid of $\smallcat{2}$. 
	Then $\zeta$ is projection onto $\im\,\zeta$.
	Hence $\im\,(\id_{\dom\,\zeta}-\pi_{\kernel\,\zeta})=\im\,\zeta$. 
	Hence $\id_{\dom\,\zeta}-\pi_{\kernel\,\zeta}=\pi_{\im\,\zeta}$ because each morphism in $M_{\dom\,\zeta}$, whose set of morphisms is partially ordered by inclusions of images [Lemma \ref{lem:commute}], is determined by its image. 
	And $\zeta_{\mathfrak{iso}}=\zeta_{\mathfrak{iso}}^2$ because $\zeta_{\mathfrak{iso}}(v)=\zeta(v)=\zeta^2(v)=\zeta^2_{\mathfrak{iso}}(v)$ for all $v\in\im\,\zeta$.  
	Hence $\zeta_{\mathfrak{iso}}=\id_{\dom\,\zeta}$ because it is an idempotent isomorphism.  
	Hence $\zeta=\pi_{\im\,\zeta}\id_{\dom\,\zeta}\pi_{\im\,\zeta}=\pi_{\im\,\zeta}$.

	Hence $\smallcat{2}$ is inverse because all $\smallcat{2}$-morphisms have pseudo-inverses and the idempotents in all submonoids of $\smallcat{2}$ commute [Lemma \ref{lem:inverse.characterization}]. 
	Moreover, $\smallcat{2}$ contains $\smallcat{1}$ because for each $\smallcat{1}$-morphism $\zeta$, $\zeta=\pi_{\im\,\zeta}\zeta(\id_{\dom\,\zeta}-\pi_{\kernel\,\zeta})$.  
	Hence (1).  
\end{proof}

\begin{cor}
	\label{cor:decomposability}
	\Decomposability{}
\end{cor}
\begin{proof}
	(1) $\iff$ (2) [Lemma \ref{lem:quivers}].
	(1) $\iff$ (4) [Theorem \ref{thm:inverse-extendability}]
	(3) implies (1) [Proposition \ref{prop:one-way}].
	
	Assume (2).
	Then there exists a finite inverse category $\smallcat{2}$ and functors $F:\smallcat{1}\ra\smallcat{2}$ and $\nabla':\smallcat{2}\ra\VECTORSPACES_k$ such that $\nabla=\nabla'F$.
	In that case $\nabla'$ is a direct sum, and hence object-wise product, of finitely many blockcodes [Proposition \ref{prop:finite.inverse}].  
	Pullbacks along $F$ of finite object-wise products are finite object-wise products, direct sums.
	Hence (3).
\end{proof}

\begin{eg}
	\label{eg:trisection}
	Consider the following faithful diagram $\nabla$ on the left:
		
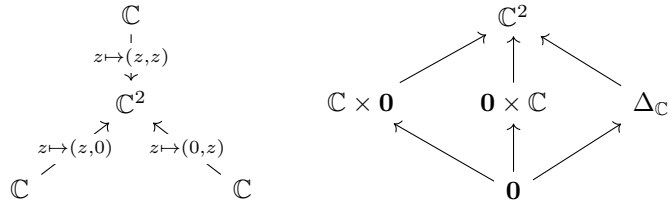
\begin{equation*}
			\begin{tikzcd}
				& \C\ar{d}[description]{z\mapsto (z,z)} \\
				& \C^2 \\
				\C\ar{ur}[description]{z\mapsto(z,0)} & & \C\ar{ul}[description]{z\mapsto(0,z)}
			\end{tikzcd}
			\quad\quad
			\begin{tikzcd}
				& \C^2\\
				\C\times{\bf 0}\ar[ur] & {\bf 0}\times\C\ar[u] & \Delta_{\C}\ar[ul] \\
				& {\bf 0}\ar[ul]\ar[u]\ar[ur]
			\end{tikzcd}
		\end{equation*}
	The Hasse diagram for $\flag_{\nabla}(\C^2)$ is depicted on the right. 
	Then
  \begin{align*}
		\mu_{\flag_{\nabla}(\C^2)}(\C^2)&=1\\
		\mu_{\flag_{\nabla}(\C^2)}(\Delta_{\C})=\mu_{\flag_{\nabla}(\C^2)}(\C\times{\bf 0})=\mu_{\flag_{\nabla}(\C^2)}({\bf 0}\times\C)&=-1\\
		\mu_{\flag_{\nabla}(\C^2)}({\bf 0})&=1
	\end{align*}
	Therefore since $1\times 2-1\times 1-1\times 1-1\times 1+1\times 0=-1<0$, $\nabla$ does not factor through an inverse category, and in particular there exist no choices of inner products on the vector spaces making $\nabla$ a diagram of inner product spaces and partial isometries much less a decomposition of $\nabla$ into a direct sum of blockcodes.
\end{eg}

\begin{eg}
	\label{eg:bisection}
	Consider the following faithful diagram $\nabla$ on the left:
		
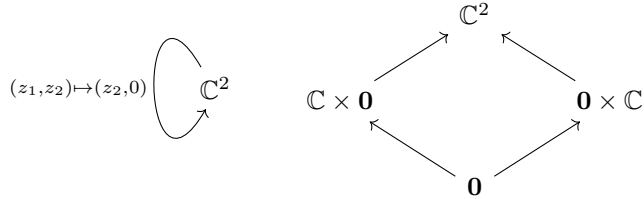
\begin{equation*}
			\begin{tikzcd}
				& \C^2\ar[out=480,in=600,loop,swap,"{(z_1,z_2)\mapsto(z_2,0)}"]
			\end{tikzcd}
			\quad\quad
			\begin{tikzcd}
				& \C^2\\
				\C\times{\bf 0}\ar[ur] & & {\bf 0}\times\C\ar[ul] \\
				& {\bf 0}\ar[ul]\ar[ur]
			\end{tikzcd}
		\end{equation*}
	The Hasse diagram for $\flag_{\nabla}(\C^2)$ is depicted on the right.  
	Then
  \begin{align*}
		\mu_{\flag_{\nabla}(\C^2)}(\C^2)&=1\\
		\mu_{\flag_{\nabla}(\C^2)}(\C\times{\bf 0})=\mu_{\flag_{\nabla}(\C^2)}({\bf 0}\times\C)&=-1\\
		\mu_{\flag_{\nabla}(\C^2)}({\bf 0})&=1
	\end{align*}
	Then $1\times 2-1\times 1-1\times 1+1\times 0=0\geqslant 0$.  
	The other requisite sums in the factorizability criterion in Theorem \ref{thm:inverse-extendability} are trivially seen to be non-negative.
	Thus $\nabla$ does factor through an inverse category.
	In fact, this inverse category can be realized as a category of partial isometries, since the above endomorphism is a partial isometry when $\C^2$ is equipped with its standard Hermitian product.
	However, $\nabla$ is indecomposable and does not decompose into a direct sum of blockcodes. 
\end{eg}

\section{Acknowledgements}
The authors greatly appreciate numerous comments and corrections from Ivo Herzog and suggestions from Facundo Memoli.  
The first author was supported by AFOSR grant FA9550-16-1-0212.



\bibliography{semigroups,simplicial,representations,ditopology,geometry,categories}{}
\bibliographystyle{plain}

\end{document}

%% file: packages.tex
\usepackage{mathrsfs}
\usepackage{mathtools} 
\usepackage{calligra}
\usepackage{amsthm}
\usepackage{amscd}
\usepackage{amssymb}
\usepackage{latexsym}
\usepackage{graphicx}
\usepackage{stmaryrd}
\usepackage{xypic}
\xyoption{curve}
\usepackage{nicefrac}
\usepackage{wasysym}
\usepackage{enumitem}
\usepackage{setspace}
\usepackage{ifthen}
\usepackage{verbatim}
\usepackage{eufrak}
\usepackage{relsize} 
\usepackage{tikz-cd}
\usepackage{stackengine} 
\usepackage{relsize} 
\usepackage[margin=1.5in]{geometry}

%% file: environments.tex
\newtheorem{thm}{Theorem}[section]
\newtheorem{cor}[thm]{Corollary}
\newtheorem{lem}[thm]{Lemma}
\newtheorem{prop}[thm]{Proposition}

\newtheorem*{lem*}{Lemma}

\theoremstyle{definition}
\newtheorem{defn}[thm]{Definition}

\newtheorem{eg}[thm]{Example}